\documentclass[12pt]{article}
\usepackage{amsmath}
\usepackage{authblk}
\usepackage{amsthm, textcomp}
\usepackage{bm}
\usepackage{amsfonts}
\usepackage{marginnote}
\usepackage{graphicx}
\usepackage{latexsym}
\usepackage[left=2cm, right=2cm, top = 1.5cm, bottom=2 cm]{geometry}
\usepackage[english]{babel}
\usepackage{amssymb}
\usepackage[utf8]{inputenc}
 \usepackage{lipsum}
\makeatletter
\g@addto@macro{\endabstract}{\@setabstract}
\newcommand{\authorfootnotes}{\renewcommand\thefootnote{\@fnsymbol\c@footnote}}%
\makeatother

\usepackage{amsmath,amsthm,amsfonts,amssymb}

\usepackage{epsfig}
\usepackage{hyperref}
\usepackage{amssymb}
\usepackage{xcolor}
\usepackage[all]{xy}
\xyoption{poly}
\usepackage{fancyhdr}
\usepackage{wrapfig}
\usepackage[T1]{fontenc}
\usepackage{amssymb,amsmath,amsthm}
\usepackage{amsmath,amssymb,amsthm}
\usepackage{graphicx}

\newcommand{\ncmd}{\newcommand}

\newtheorem{theorem}{Theorem}
\newtheorem{lemma}{Lemma}
\newtheorem{proposition}{Proposition}

\theoremstyle{definition}
\newtheorem*{definition}{Definition}
\newtheorem{remark}{Remark}
\ncmd{\E}{\mathbb{E}}
\ncmd{\Oc}{\mathbb{O}}
\ncmd{\Ha}{\mathbb{H}}
\ncmd{\R}{\mathbb{R}}
\ncmd{\C}{\mathbb{C}}
\ncmd{\Z}{\mathbb{Z}}
\ncmd{\N}{\mathbb{N}}
\ncmd{\Sph}{\mathbb{S}}
\ncmd{\T}{\mathbb{T}}
\ncmd{\D}{\mathbb{D}}
\ncmd{\re}{\mathrm{Re}}
\ncmd{\im}{\mathrm{Im}}
\ncmd{\sing}{\mathrm{sing}}
\ncmd{\reg}{\mathrm{reg}}
\ncmd{\red}{\mathrm{red}}
\ncmd{\bs}{\backslash}
\ncmd{\ov}{\overline}
\ncmd{\noi}{\noindent}
\ncmd{\di}{\displaystyle}
\ncmd{\ra}{\rightarrow}
\ncmd{\lra}{\longrightarrow}
\title{Epicycles in the hyperbolic sky}
\author[1]{Olga Romaskevich \\email olga@pa-ro.net, olga.romaskevich@univ-rennes1.fr\\Univ Rennes, UR1, CNRS, IRMAR - UMR 6625, F-35000 Rennes, France}
\date{}
\begin{document}
\maketitle

\emph{In memory of my grandfather \\
			V.V. Beletskii, \\
      a mathematician and a poet}
\begin{center}
\textbf{Abstract}
\end{center}
\small
Consider a swiveling arm on an oriented complete riemannian surface composed of three geodesic intervals, attached one to another in a chain. Each interval of the arm rotates with constant angular velocity around its extremity contributing to a common motion of the arm. Does the extremity of such a chain have an asymptotic velocity ? This question for the motion in the euclidian plane, formulated by J.-L. Lagrange, was solved by P. Hartman, E. R. Van Kampen, A. Wintner.  We generalize their result to motions on any complete orientable surface of non-zero (and even non-constant) curvature. In particular, we give the answer to Lagrange's question for the movement of a swiveling arm on the hyperbolic plane. The question we study here can be seen as a dream about celestial mechanics on any riemannian surface : how many turns around the Sun a satellite of a planet in the heliocentric epicycle model would make in one billion years ? 
\normalsize

\bigskip

\noindent

\begin{center}
\textbf{Acknowledgments}
\end{center}
I am grateful to Anatoly Stepin for sharing with me the question about the asymptotic angular velocity of a swiveling arm on the hyperbolic plane when I was a student at Moscow State University. I thank \'{E}tienne Ghys for very fruitful discussions that helped me change the approach of this question and drastically simplify the arguments. I also thank Bruno Sevennec as well as the anonymous referee for pertinent questions and remarks that helped me improve the text. The principal part of this work was accomplished when I was a graduate student at the UMPA laboratory at \'{E}cole Normale Sup\'{e}rieure de Lyon. I thank my reporters, François Beguin and Alain Chenciner, for their comments. During the period of the work on this project, I was supported by the LABEX MILYON (ANR-10-LABX-0070) of Universit\'{e} de Lyon, within the program "Investissements d'Avenir" (ANR-11-IDEX-0007) operated by the French National Research Agency (ANR) as well as by a personal grant l'Or\'{e}al-UNESCO for Women in Science 2016. 
\newpage
%

As far as we know, the first models of our planetary system started appearing in the 4th century BC in Greece although the evidence of astronomical observations goes back to the 16th century BC in Babylon. At the end of the 3rd century BC Apollonius of Perga proposed a following geocentric model of the movement. All the planets are following the trajectories which correspond to the sum of two circular movements. First, each of the planets is moving around a corresponding point by forming a circle which is called an \emph{epicycle}. The centers of these epicycles are not fixed but also moving around some point close to the Earth on the bigger circles called \emph{deferents}. Both of these circular movements (of planets on epicycles as well as of the centers of epicycles on deferents) are done with constant angular velocity. This model is called a \emph{geocentric epicycle model} of planetary motion. Of course, one can add the third (fourth, fifth, etc.) set of circles in a similar way in order to obtain the epicycle model for the satellites.

Greeks firmly believed that all movement can be described as a sum of perfect circular movements and the epicycle model of planetary motion is one of the mutliple theories based on that belief. The idea of the decomposition of a movement in a sum of circular ones, one can speculate, finds its place in mathematics much later, in Fourier decomposition of a function into the sum of exponentials with different frequencies, see \cite{Etienne} for more discussion. The epicycle model was improved and largely used by Hipparchus of Rhodes, and, a couple of centuries later, by Ptolemy. 

The problem that we consider here was formulated by Joseph-Louis Lagrange \cite{Lagrange} much later, in the XVIIIth century. He has  also been studying planetary motion but his model was very different from that of Ptolemy since he was working with the force of gravity that Ptolemy had no idea of. Lagrange started with the $N$-body problem: a set of $N$ bodies is moving in the space with respect to gravitational forces that the bodies exercise on each other. This model was first defined by Newton. 

While studying $N$-body problem, Lagrange was interested in  the variation of the longitude of the perihelion for the orbit of a planet in such a system. Surprisingly, the approximation of this variation in this difficult problem boils down to a much simpler problem - the study of the \emph{asymptotic velocity of a satellite in the heliocentric epicycle model}. The heliocentric epicycle model is equivalent to the geocentric model defined above but with the Sun in the center of the system instead of the Earth. The reader will soon see that the problem we study below - the \emph{Lagrange problem} - is nothing more than indeed the study of the asymptotic velocity of a satellite in the heliocentric epicycle model. Ptolemy would have liked it. 

Throughout this paper, the movement of planets will be studied not in the euclidian but in the hyperbolic world. And also, in the approach of this article, the radii of the epicycles are not necessarily considered small with respect to the radius of the deferent. These may seem as one and then another unrealistic assumptions on the movement. As Henri Poincaré writes in the chapter on Astronomy in the \emph{Value of Science} \cite{HP}, \emph{if we are making assumptions, one more assumption won't cost us much.} \footnote{\emph{puisque nous sommes en train de faire des hypothèses, une hypothèse de plus ne nous coûtera pas davantage.}} As far as for utility, Poincaré gives a point of view on astronomy which also, we think, applies to fundamental mathematics.

\smallskip

\emph{Astronomy is useful because it raises us above ourselves; it is useful because it is grand; it is useful because it is beautiful; that is what we should say.}
\footnote{\emph{
L'astronomie est utile, parce qu'elle nous élève au-dessus de nous-mêmes ; elle est utile, parce qu'elle est grande; elle est utile parce qu'elle est belle ; voilà ce qu'il faut dire.}}

\section{Lagrange epicycle problem.}\label{sec:1}
\subsection{Statement of the problem.}\label{subsec:1.1}
\begin{definition}
For the fixed numbers $l_1, \ldots, l_N \in \mathbb{R}_{+}$ consider a map $\Psi$ from the $N$-torus to the complex plane that sends a point $\bm{\theta}=(\theta_1, \ldots, \theta_N) \in \mathbb{T}^N = \mathbb{R}^N / \left( 2 \pi \mathbb{Z} \right)^N$ to the point 
\begin{equation}
\label{eq:firstone}
\Psi(\bm{\theta})=\sum_{j=1}^N l_j e^{i \theta_j}.
\end{equation}
We will call $\Psi$ \emph{a swiveling arm} of type $l=(l_1, \ldots, l_N)$ on the complex plane, see Figure \ref{fig:swiveling_arm}. The intervals connecting the points $0$ and $l_1 e^{i \theta_1}$ as well as $\sum_{j=1}^k l_j e^{i \theta_j}$ and $\sum_{j=1}^{k+1} l_j e^{i \theta_j}, k=1, \ldots, N-1$ are called the \emph{joints} of the swiveling arm.
\end{definition}

The topology of $\Psi^{-1}(z)$ for some fixed $z\in \mathbb{C}$ is an interesting question, considered, among others, by J.-C. Haussmann in \cite{Haus1, Haus2}, M. Kapovich and J. Millson \cite{KapMil} and D. Zvonkine\cite{Z}. We will add a simple linear motion to this geometrical construction in a following way.

\begin{figure}
\centering
\includegraphics*[scale=1]{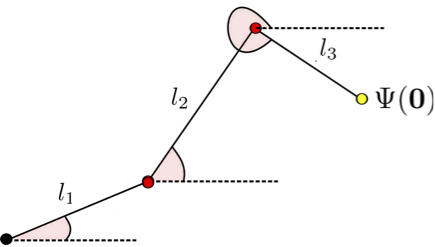}
\caption[]{A swiveling arm of type $(l_1, l_2, l_3)$ on the plane. The numbers $l_j$ are the lengths of the intervals in the arm.The angles $\theta_j$ in the equation \eqref{eq:firstone} correspond to the angles that the intervals make with the horizontal direction in a plane in the position $\Psi(\bm{\theta})$.
 }\label{fig:swiveling_arm}
\smallskip
\end{figure}

Fix $N$ real numbers $\omega_1, \ldots, \omega_N \in \mathbb{R}$ and consider a flow $g^t$ of the following constant vector field $X$ on the torus $\mathbb{T}^N$:
\begin{equation}\label{eq:vfield}
X=\sum_{j=1}^{N} \omega_j \frac{\partial}{\partial \theta_j},  \omega_j \in \mathbb{R}.
\end{equation}
Then, a function $z(t)=\Psi \circ g^t (\bm{\theta}): \mathbb{R}_+ \rightarrow \mathbb{C}$ defines a curve on the plane. The question of Lagrange was the following : does $z(t)$ have an asymptotic angular velocity and if yes, is it possible to calculate its value as a function of the parameters $l_j$ and $\omega_j$?

\begin{definition}
Consider a curve \begin{equation}\label{eq:z}
z(t)=\Psi \circ g^t (\bm{\theta}), z : \mathbb{R}_+ \rightarrow \mathbb{C},
\end{equation}
where $\Psi, T^t$ defined above by \eqref{eq:firstone} and \eqref{eq:vfield} and let $\varphi(t)$ be a continuous branch of the argument $\arg z(t)$. Then the \emph{Lagrange problem on the plane} is a question of studying the limit (first, the question of its existence and then, its numerical value)
\begin{equation}\label{eq:angular}
\omega:=\lim_{t \rightarrow \infty} \frac{\varphi(t)}{t}
\end{equation}
as a function of parameters $l_j \in \mathbb{R}_+, \omega_j \in \mathbb{R}$ and initial conditions $\bm{\theta} \in \mathbb{T}^N$. 
We call this limit \emph{asymptotic angular velocity} $\omega$. 
\end{definition}

\begin{remark}\label{remark:passage_through_zero} \textbf{Passage through zero of} $z(t)$ \textbf{.}
Note that at some moments of time $t$ the function $z(t)$ may happen to be $0$. Those are the moments when the swiveling arm closes up into a polygon (possibly, a self-intersecting one). If this happens, the continuous branch of the argument in the definition \eqref{eq:angular} of $\omega$ as it is given can't be chosen. But one can remedy to this fact : the function $\arg z(t)$ can be made continuous along the curve $z(t)$ (see explanations below). In what follows we place ourselves in this setting.
\end{remark}

First of all, if the set $\mathcal{N}=\left\{t \in \mathbb{R}_+ | z(t)=0\right\}$ is finite, then the limit \eqref{eq:angular} is obviously well defined. Even if the cardinality of the set $\mathcal{N}$ is infinite, it is still a discrete set
(by analiticity of $z$). 

Since $z(t)$ is an analytic function, in the neighborhood $U$ of its zero $t_0 \in \mathcal{N}$ we can write $z(t)=a(t)(t-t_0)$ for some $a(t)$ analytic and $a(t) \neq 0$ for $t \in U$, and $k>0, k \in \mathbb{Z}$. We are searching for a continuous solution $\varphi(t)$ of the equation 
\begin{equation}\label{eq:basic_form}
z(t)=r(t) e^{i \varphi(t)}.
\end{equation}

 For this, we set $r(t):=s |a(t)| (t-t_0)^k$ where $s=\pm 1$ is determined by the choice of the sign of $r(t),$ for $t<t_0$. Then 

\begin{equation*}
\exp (i \varphi(t)) = \frac{z(t)}{r(t)}= s \frac{a(t)}{|a(t)|}
\end{equation*}

is a well-defined function in the vicinity of $t_0$. It is determined for $t>t_0$ modulo $2 \pi$ and also as a continuous real analytic function by its values for $t<t_0$. Hence,  by induction based on the discreteness of $\mathcal{N}$, one can conclude that $\varphi(t)$ is determined by its initial value $\varphi(0)$. 

Another way of looking at $\varphi(t)$ is to say that it is defined as an integral :
\begin{equation}\label{eq:right-hand_side}
\varphi(t)=\varphi(0)+\int_0^t \mathrm{Im} \frac{z'(s)}{z(s)} ds.
\end{equation}
The right-hand side of \eqref{eq:right-hand_side} is well-defined and analytic in the vicinity of any $t_0 \in \mathcal{N}$(and hence, everywhere), since in $U$ one can write
\begin{equation}\label{eq:second_term}
\frac{z'(t)}{z(t)}=\frac{a'(t)}{a(t)}+\frac{k}{(t-t_0)},
\end{equation}
and the second term in \eqref{eq:second_term} is real.

The function $\varphi(t)$ is an only solution of the equation \eqref{eq:basic_form} with $\varphi(t)$ differentiable (and hence, $r(t)$ as well). In the equation $\frac{z'}{z}=\frac{r'}{r}+i \varphi'$ the first term on the right-hand side diverges but it doesn't count when one takes the imaginary parts.

We will use the integral representation \eqref{eq:right-hand_side} in a crucial way in Section \ref{sec:2}, and along the paper.

\bigskip
Lagrange himself considered only the simplest case of this problem of a swiveling arm with two joints, $N=2$. He proved that in a linear motion described above the longer interval "wins"~:  the limit angular velocity exists, doesn't depend on an initial condition $\bm{\theta} \in \mathbb{T}^2$ and is equal to the angular velocity of the longer interval. That is, if $l_1>l_2$ then $\omega=\omega_1$ and vice-versa, for $l_1<l_2$ we have $\omega=\omega_2$. In the case of equal lengths $l_1=l_2$ a direct computation gives $\omega=\frac{1}{2}(\omega_1+\omega_2)$ taking into account the remark \ref{remark:passage_through_zero} above.

The argument of Lagrange can be easily generalized for any $N$ to the case when the length of one of the intervals (say, the one with the index $j$) is bigger than the sum of the lengths of all other intervals.
Then the limit angular velocity $\omega$ exists and $\omega=\omega_j$. Even more, the continuous branch $\varphi(t)$ of the function $\arg z(t)$ has a linear asymptotic behavior $\varphi(t) = \omega_j t + O(1)$ when $t\rightarrow \infty$, \cite{BorgeJessen}. This case being quite simple, things do get much more complicated if the lengths of the intervals are comparable.

\subsection{Historical remarks and our motivation.}\label{subsec:1.2}
Suppose that the number of intervals $N$ as well as their angular velocities and lengths $\omega_j \in \mathbb{R}, l_j \in \mathbb{R}_+, j \in [[1,n]]$, and also initial conditions $\bm{\theta} \in \mathbb{T}^N$ are arbitrary. In this case, the question of the existence of limit angular velocity $\omega$ for Lagrange problem is quite tricky. As Lagrange writes in \cite{Lagrange}, "\emph{Il est fort difficile et peut-être même impossible de se prononcer, en général, sur la nature de l'angle} $\varphi$"\footnote{"\emph{It is hard and maybe even impossible to say something on the nature of angle $\varphi$ in the general case}"
(English translation). Lagrange's angle $\varphi$ is the continuous branch of the argument $\varphi(t)$ defined above.
 }. In 1945, following the works of P. Bohl \cite{Bohl}, B. Jessen and H. Tornehave have proven together the existence of this limit for any initial data. But we still do not know how to write out $\omega$ as a function of this data $\omega=\omega\left(\omega_j, l_j, \bm{\theta}\right)$ although some asymptotic estimates exist, see \cite{JT} for the survey of the question.

In general, the asymptotic angular velocity $\omega$ depends on initial conditions $\bm{\theta} \in \mathbb{T}^N$. Although in the case when angular velocities $\omega_j, j=1, \ldots, N$ are independent over $\mathbb{Q}$, it does not. The key idea is to replace the time average \eqref{eq:angular} by the space average. In 1937 P. Hartman, E. R. Van Kampen and A. Wintner \cite{HKW} elegantly used this idea. They provided a calculation that gave an expression for $\omega$ as a linear combination of $\omega_j$ with coefficients given by some explicit space integrals. Note that Birkhoff's ergodic theorem which is now classical, appeared just six years before the work of Hartman - Van Kampen - Wintner. H.Weyl did a considerable work in order to fill in all the technical details. In his 1938 article \cite{Weyl} Weyl explains why the ergodic theorem can be applied in the Hartman-van Kampen-Wintner case. The argument of Weyl is mostly topological. 

Of course, the rational independence of $\omega_j$ is crucial in the arguments since only in this case the flow of the vector field \eqref{eq:vfield} is ergodic. The Hartman-van Kampen-Wintner-Weyl result gives a very beautiful geometric answer to the Lagrange problem in the case when the number of joints is equal to three.

\begin{theorem} [P. Hartman, E. R. Van Kampen and A. Wintner, H. Weyl] \cite{HKW, Weyl, KSF} \label{thm:euclidian_3}
Consider the dynamics of a swiveling arm of type $l=(l_1, l_2, l_3)$ with angular velocities $\omega_j, j=1,2,3$ governed by a vector field \eqref{eq:vfield}, and a corresponding Lagrange problem on the plane. Suppose that $l_j$ satisfy all of three strict triangle inequalities and $\omega_j$ are rationally independent. Then the asymptotic angular velocity $\omega$ exists, doesn't depend on the initial condition $\bm{\theta} \in \mathbb{T}^3$ and is equal to the convex sum
\begin{equation}\label{eq:finalanswer}
\omega=\frac{\alpha_1}{\pi} \omega_1+\frac{\alpha_2}{\pi} \omega_2+\frac{\alpha_3}{\pi} \omega_3,
\end{equation}
where $\alpha_j$ are the angles in the triangle formed by intervals with sides $l_j, j=1,2,3$. The angle $\alpha_j>0$ is the angle opposite to the side of the length $l_j, j=1,2,3$, see Figure \ref{fig:swiveling_three_triangle}. 
\end{theorem}

\begin{figure}
\centering
\includegraphics*[scale=0.8]{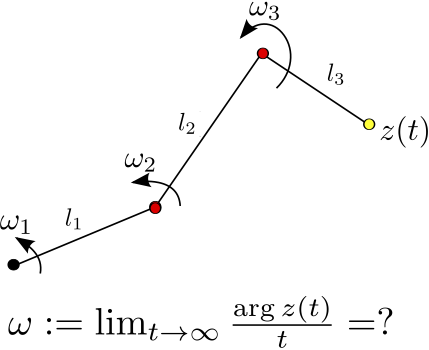}
\includegraphics*[scale=0.8]{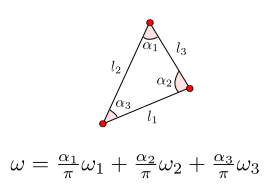}

\caption[]{Lagrange problem on the plane in Hartman-van Kampen-Wintner case when  $\omega_1, \omega_2, \omega_3$ are rationally independent. The asymptotic angular velocity $\omega$ is equal to a convex sum of $\omega_j$ with coefficients which are proportional to the angles $\alpha_j, j=1,2,3$ in the triangle which is constructed from the joints of the system. }\label{fig:swiveling_three_triangle}
\smallskip
\end{figure}

\begin{remark}
So one can see that still, in some way, a longer interval "wins": its angular velocity will be taken in a convex sum with a bigger coefficient.
\end{remark}
 
The initial motivation for us was to understand the Lagrange problem on the hyperbolic plane. We will give proper definitions in Section \ref{sec:3} but the reader can easily make her opinion about that since the definition of Lagrange problem actually uses only the concepts of intervals (geodesic segments) and angles between lines, present in any geometry.

A straightforward translation of the proof of Theorem \ref{thm:euclidian_3} from \cite{HKW} for the hyperbolic geometry is possible but involves lots of quite tedious double integrals computation. Our goal was to extract all geometrical ideas from the initial proof of the Theorem \ref{thm:euclidian_3} in order to find a new proof which will be easily translated to the hyperbolic case, without computation. 

\subsection{Plan of the paper.}\label{subsec:1.3}
In Section \ref{sec:2} we remind the reader the classical proof of Theorem \ref{thm:euclidian_3} that we repeat from its wonderful exposition in \cite{KSF} by adding the technical details. In Section \ref{sec:3} we present a new way of looking at the Largange problem (see Subsection \ref{subsec:3.1}) and give a new proof of the same Theorem \ref{thm:euclidian_3}, see Subsection \ref{subsec:3.2}. In Section \ref{sec:4} we adapt out proof from Section \ref{sec:3} for constant curvature geometries (Subsection \ref{subsec:4.1}) as well as for non-constant (but close to constant) curvature geometries (Subsection \ref{subsec:4.2}).
\bigskip

\textbf{Setting.} 
From now on and till the end of the article, we will consider the case of a swiveling arm with three joints such that the lengths of the joints $l_j,j=1,2,3$ verify all three of strict triangle inequalities $l_1 < l_2+l_3, l_2 < l_3+l_1$ and $l_3 < l_1+l_2$. In other words, there is no dominating interval whose length is bigger than the sum of the two other lengths. 

\section{Classical proof.}\label{sec:2}
In this Section we will remind a reader of the proof of a generalization of the Theorem \ref{thm:euclidian_3} for the case of a sziveling arm with $N$ joints on the plane.

\begin{theorem} [P. Hartman, E. R. Van Kampen and A. Wintner, H. Weyl] \cite{HKW, Weyl, KSF} \label{thm:euclidian_n}
Consider the dynamics of a swiveling arm of type $l=(l_1, l_2, \ldots, l_N)$ governed by a vector field \eqref{eq:vfield} with the angular velocities of joints $\omega_1, \omega_2, \ldots, \omega_N$ independent over $\mathbb{Q}$. Suppose also that $l_j \in \mathbb{R}_+$ are such that for all vectors of signs $\bm{\varepsilon}=(\varepsilon_1, \ldots, \varepsilon_N) \in \{-1, 1\}^{N}$ the signed sum of the lengths $l_j$ is not equal to zero:
\begin{equation}\label{eq:signed_sum}
\sum_{j=1}^N  \varepsilon_j l_j\neq 0.
\end{equation}
Then the solution $\omega$ of Lagrange problem on the plane exists, doesn't depend on the initial condition $\bm{\theta} \in \mathbb{T}^N$ and 
\begin{equation*}
\omega=q_1 \omega_1 + \ldots + q_N \omega_N,
\end{equation*}

where $q_k \in [0,1], k=1, \ldots, N$ are equal to the volumes of the subsets of the torus $\mathbb{T}^N$ and are defined as follows:

\begin{equation*}
q_k = \mathrm{mes}_N \left\{ \bm{\theta}=(\theta_1, \ldots, \theta_N) \in \mathbb{T}^N \left| \right.
| l_1 e^{i \theta_1}+ \ldots+ l_{k-1} e^{i \theta_{k-1}}+l_{k+1} e^{i \theta_{k+1}}+\ldots +l_N e^{i \theta_N}| < l_k
\right\}.
\end{equation*}
Here $\mathrm{mes}_N$ is the normalized Lebesgue measure on the torus $\mathbb{T}^N$.
\end{theorem}

The additional condition \eqref{eq:signed_sum} in the formulation of the Theorem is motivated by the following

\begin{proposition}\label{prop:submersion}
Consider a swiveling arm of type $(l_1, \ldots, l_N)$ on the plane. Then, the map $\Psi: \mathbb{T}^N \rightarrow \mathbb{C}$ defined by \eqref{eq:firstone} in restriction to $\Psi^{-1}(0)$ is a submersion if and only if the condition \eqref{eq:signed_sum} holds.
\end{proposition}
\begin{proof}
By calculating explicitly the differential $d \Psi_{\bm{\theta}}: \mathbb{R}^N \rightarrow \mathbb{C} \simeq \mathbb{R}^2$ we obtain 
\begin{equation*}
d \Psi_{\bm{\theta}}=i \left(l_1 e^{i \theta_1}, \ldots, l_N e^{i \theta_N}\right).
\end{equation*}
This $2 \times N$ matrix has its rank smaller than $2$ if and only if the complex numbers $l_j e^{i \theta_j}$ are all $\mathbb{R}$–proportional, in other words the corresponding vectors lie on the same line passing by $0 \in \mathbb{C}$. One considers the restriction $\Psi|_{\Psi^{-1}(0)}$. Conditions $\mathrm{rk} \; d \Psi_{\bm{\theta}} <2$ and $\Psi(\bm{\theta})=0$ together are equivalent to the existence of the coefficients $\varepsilon_j \in \{-1,1\}$ such that $\sum_{j} l_j \varepsilon_j=0$ with $e^{i \theta_j}=\varepsilon_j$. Hence $\theta_j=0$ (if $\varepsilon_j=1$) or $\theta_j=\pi$ (if $\varepsilon_j=-1$).
\end{proof}

\begin{remark}
Before starting a proof of the Theorem \ref{thm:euclidian_n}, let us first notice that it implies Theorem \ref{thm:euclidian_3}.
First, let us note that condition \eqref{eq:signed_sum} holds true for $l_j$ that satisfy all three triangle inequalities: the triangle with sides $l_j$ is a rigid form that can't be flattened into a line.


Theorem \ref{thm:euclidian_n} gives

\begin{equation}\label{eq:insider}
q_3=\mathrm{mes}_2 \left\{ (\theta_1, \theta_2) \in \mathbb{T}^2 \left| \right.
| l_1 e^{i \theta_1}+ l_{2} e^{i \theta_2}| < l_3
\right\}.
\end{equation}

For any fixed $\theta_1$ one can easily see (as on the Figure \ref{pic:explanation}) that the measure in question is equal to $\frac{\alpha_3}{\pi}$ (after renormalizing), i.e. it doesn't depend on $\theta_1$. Then the integration with respect to $\theta_1$ will give $q_3=\frac{\alpha_3}{\pi}$. Because of the symmetry of the answer with respect to the exchange of the sides, we get the final answer \eqref{eq:finalanswer}. 
\begin{figure}
\centering
\includegraphics*[scale=0.9]{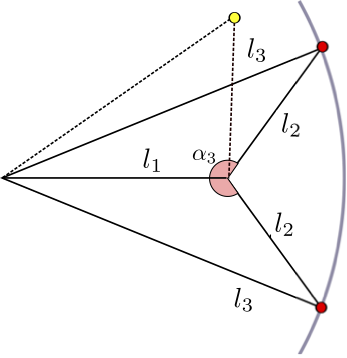}
\caption[]{Let us fix some value $\theta_1$ corresponding to the position of the first interval (here $\theta_1=0$). On the picture one can see the geometrical meaning of the set appearing in \eqref{eq:insider}. The angles $\theta_2$ which give the points $(\theta_1, \theta_2)$ inside this set correspond to the position of the second joint such that the sum $l_1 e^{i \theta_1}+ l_{2} e^{i \theta_2}$ stays inside the circle of radius $l_3$. These positions are marked by the angle range in the interval  $\theta_2 \in (-\alpha_3, \alpha_3)$. The two "boundary" positions are those that correspond to the moments $\bm{\theta}$ when $\Psi(\bm{\theta})=0$. These moments are the moments when a swiveling arm closes up into a triangle. }\label{pic:explanation}
\smallskip
\end{figure}
\end{remark}

\begin{lemma}[A rotating system of coordinates]\label{lemma:rotating}
Suppose that the limit asymptotic velocity in the Lagrange problem exists for the dynamics of a swiveling arm of type $(l_1, l_2, l_3)$ with angular velocities $\omega'_1=0, \omega'_2=\omega_2-\omega_1$ and $\omega'_3=\omega_3-\omega_1$ and is equal to $\omega$. Then the limit asymptotic velocity exists as well for the dynamics of a swiveling arm of the same type with angular velocities $\omega_1, \omega_2, \omega_3$ and is equal to $\omega_1+\omega$.
%
\end{lemma}

\begin{proof}
\normalfont
The two systems described in the formulation, one with corresponding angular velocities of joints $(0, \omega_2-\omega_1, \omega_3-\omega_1)$  and another with $(\omega_1, \omega_2, \omega_3)$, are related by the rotation. Indeed, the position of the end point $z_2(t)$ of the second system at time $t$ is just the image of the position of the endpoint $z_1(t)$ for the first one under the rotation by $\omega_1 t $ around $0$. 
\end{proof}

Now we are ready to give the proof of Theorem \ref{thm:euclidian_n}. As we said above, the main ideas are all described in \cite{KSF} in a very clear and concise way but we find it useful to present this argument here for the sake of completeness and clarity.

\begin{proof}
\textbf{Step 1. Main idea: pass from the time average to the space average.}

We are interested in the asymptotic behavior of the argument of the function $z(t): \mathbb{R}_+ \rightarrow \mathbb{C}$ given by \eqref{eq:z}. Let us write out $z(t)$ in the polar form, $z(t)=r(t) \exp \varphi (t)$. A formal computation gives $\ln z(t)=\ln r(t)+i \varphi(t)$ and, by passing to a real part and then taking a derivative with respect to $t$, we obtain the expression for the derivative of the angle
\begin{equation}\label{eq:for_varphi}
\dot{\varphi}=\re \left( \frac{1}{i} \frac{\dot{z}(t)}{z(t)} \right).
\end{equation}

Here by $\varphi(t)$ we understand a continuous branch of the argument and this computation gives a valid formula at least in the case when $z(t) \neq 0 \;  \forall t \in \mathbb{R}_+$. 

The derivative $\dot{\varphi}$ is precisely the quantity that is interesting for us since the asymptotic angular velocity $\omega$ is the ratio between the increment of the angle function $\varphi(t)$ on the long period of time $T$ and $T$ itself. That can be calculated by Newton-Leibniz as 

\begin{equation}\label{eq:newomega}
\omega=\lim_{T \rightarrow \infty} \frac{\varphi(T)}{T}=\lim_{T \rightarrow \infty} \frac{1}{T}\int_0^T\dot{\varphi}(t) dt.
\end{equation}

The main idea of Hartman, van Kampen and Wintner was that instead of calculating the time average \eqref{eq:newomega}, one can transform it to the space average of some function $f$.

Indeed, let us insert in the equation \eqref{eq:for_varphi} an explicit formula \eqref{eq:firstone} for $z(t)$. We obtain

\begin{equation}\label{eq:varphi_calculation}
\dot{\varphi}= \re \left(
\frac{\sum_{j=1}^N l'_j \omega_j e^{i \omega_j t}}{\sum_{j=1}^N l'_j e^{i \omega_j t}}
\right)=
 \re \left(\frac{\sum_{j=1}^N l_j \omega_j e^{i (\omega_j t+\theta_j^{(0)})}}{\sum_{j=1}^N l_j e^{i (\omega_j t+\theta_j^{(0)})}}
\right).
\end{equation}

Here $l'_j = l_j e^{i \theta_j^{(0)}}$ where $\bm{\theta}^{(0)}=\left(\theta_1^{(0)}, \ldots, \theta_N^{(0)}\right) \in \mathbb{T}^N$ is a vector corresponding to the initial position of the swiveling arm.

Let us define $f: \mathbb{T}^N \rightarrow \overline{\mathbb{R}}$ as 

\begin{equation}\label{eq:definition_f}
f(\bm{\theta}):= \re  \left(\frac{\sum_{j=1}^N l_j \omega_j e^{i \theta_j}}{\sum_{j=1}^N l_j e^{i \theta_j}}
\right).
\end{equation}

Then, in previous notations, \eqref{eq:varphi_calculation} can be rewritten simply as $\dot{\varphi}=f\left(T^t \bm{\theta}^{(0)}\right)$ and hence $\omega$ (if it exists) is represented by the limit

\begin{equation}\label{eq:eqlimit}
\omega = \lim_{T \rightarrow \infty} \frac{1}{T} \int_0^T f \left(
g^{t} \bm{\theta}^{(0)} \right) dt.
\end{equation}

The idea is to apply the ergodic theorem for the flow $T^t$ to substitute the limit \eqref{eq:eqlimit} by the space integral in order to write $\omega=\int_{\mathbb{T}^N} f(\bm{\theta}) d\bm{\theta}$. This is actually true but now let us prove it properly: the difficulty is that the denominator in the definition \eqref{eq:definition_f} of the function $f$ explodes when $\Psi(\bm{\theta})=0$.

\bigskip

\textbf{Step 2. Justifying the use of ergodic theorem.}

First note that the function $f: \mathbb{T}^N \rightarrow R$ defined by \eqref{eq:definition_f} is integrable. Indeed, to prove this it is sufficient to prove that the function $\frac{1}{\Psi}$ is integrable. By Proposition \ref{prop:submersion}, $\Psi: \mathbb{T}^N \rightarrow \mathbb{C}$ is a submersion on $\Psi^{-1}(0)$ which has codimension $2$.
Hence in the neighborhood of any pole of $f$(equivalently, zero of $\Psi$), there is a complex chart $w \in \mathbb{C}$ on the local plane, transverse to $\Psi^{-1}(0)$ in which $\Psi(w)=w$. Hence the reciprocal $\frac{1}{\Psi}$ is integrable since $\frac{1}{|w|} \in L^1_{\mathrm{loc}} \mathbb{C}$.

%
%

The function $f$ is integrable but is not continuous since the denominator $\Psi(\bm{\theta})$ can be $0$. By averaging the function $f$ on the part of the trajectory of $g^t$ ranging from time $0$ to time $T_0 \in \mathbb{R}, T_0>0$, we get a continuous function on the torus $
\tilde{f} \in C(\mathbb{T}^N)$ :

$$\tilde{f}(\bm{\theta}):= \frac{1}{T_0} \int_0^{T_0} f \circ g^t (\bm{\theta}) dt.$$
The proof of the continuity of the function $\tilde{f}$ uses the Remark \ref{remark:passage_through_zero}. Indeed, the curve $g^t (\bm{\theta})$ for $t\in (0,T_0)$ and the analogical curve for a close $\theta$ have the property that one of them goes through zero and another doesn't but the argument change is the same (modulo $\pi$) since it is defined by the change of the slope of a tangent line to such a curve. Note also that the time averages as well as space averages of the functions $f$ and $\tilde{f}$ coincide.

Indeed, for the space averages since $g^t$ is a measure-preserving flow,

\begin{multline}
\int_{\mathbb{T}^N} \tilde{f} (\bm{\theta})= \int_{\mathbb{T}^N} \frac{1}{T_0} \int_{0}^{T_0} f \circ g^t (\bm{\theta}) dt d\bm{\theta} = \int_0^{T_0} \frac{1}{T_0} \int_{\mathbb{T}^N} f \circ g^t (\bm{\theta}) d \bm{\theta} dt = \\= \int_0^{T_0} \frac{1}{T_0} \int_{\mathbb{T}^N} f(\theta) d\bm{\theta} dt= \int_{\mathbb{T}^N} f(\bm{\theta}) d\bm{\theta}.
\end{multline}

And for the time averages $f_{\infty}(\bm{\theta})$ and $\tilde{f}_{\infty}(\bm{\theta})$, analogously, we get

\begin{multline}
\tilde{f}_{\infty}(\bm{\theta}):=\lim_{T \rightarrow \infty} \frac{1}{T} \int_0^T \tilde{f} \circ g^t(\bm{\theta}) dt=\lim_{T \rightarrow \infty} \frac{1}{T} \int_0^T \frac{1}{T_0} \int_0^{T_0} f \circ g^{t+\tau} (\bm{\theta}) d\tau d t =\\= \frac{1}{T_0} \int_0^{T_0} \lim_{T \rightarrow \infty} \frac{1}{T} \int_{0}^T f \circ g^{t+\tau} (\bm{\theta}) dt d \tau = f_{\infty} (\bm{\theta}).
\end{multline}

Note that the flow $g^t$ is uniquely ergodic (since $\omega_j$ are rationally independent
\footnote{The same assumptions about $\omega_j$ hold for the Theorem \ref{thm:constant_curvature} 
and the swiveling arm on the hyperbolic plane.}) 
and $\tilde{f} \in C(\mathbb{T}^n)$ hence the space averages of $\tilde{f}$ coincide with time averages of $\tilde{f}$ for all (and not only almost all) values  of $\theta \in \mathbb{T}^N$. Hence the same is true for the function $f$ and the limit \eqref{eq:eqlimit} can be written as a space average for all ${\theta} \in \mathbb{T}^N$. Hence we obtain that the limit 
for any initial position of the swiveling arm $z(0) \in \mathbb{C}$ is just given by the space integral that can be explicitly calculated.

\bigskip

\textbf{Step 3. Calculation.}

Denote  $B_j:=B(\theta_1,  \ldots, \theta_{j-1}, \theta_{j+1}, \ldots, \theta_N):= \Psi(\bm{\theta})-l_j e^{i \theta_j}$. This quantity doesn't depend on $\theta_j$. Then,
\begin{multline*}
\int_{\mathbb{T}^N} f(\theta) d\theta= \re \int_{\mathbb{T}^N} \frac{\sum_{j} \omega_j l_j e^{i \theta_j}}{\sum_j l_j e^{i \theta_j}} d\theta_1 \ldots d\theta_N = 
\sum_{j=1}^{N} \omega_j l_j \re \int_{\mathbb{T}^N} \frac{e^{i \theta_j} d\theta_1 \ldots d\theta_N}{\sum_j l_j e^{i \theta_j}}=\\
=\sum_{j=1}^N \omega_j l_j \re \int_{\mathbb{T}^{N-1}} \int_0^{2 \pi} \frac{e^{i \theta_j} d\theta_j}{l_j e^{i \theta_j}+B(\theta_1, \ldots, \theta_{j-1}, \theta_{j+1}, \ldots, \theta_N)} d\theta_1 \ldots  d\theta_{j-1} d\theta_{j+1} \ldots d\theta_N=\\
=\sum_{j=1}^N \omega_j l_j \re \int_{\mathbb{T}^{N-1}} \int_0^{2 \pi} \frac{1}{i l_j} \frac{\partial \ln (B_j + l_j e^{i \theta_j})}{\partial \theta_j} d\theta_1 \ldots  d\theta_{j-1} d\theta_{j+1} \ldots d\theta_N=\\
\sum_{j=1}^N \omega_j \re \int_{\mathbb{T}^{N-1}} \int_0^{2 \pi}   \frac{1}{i} \frac{\partial \ln (B_j + l_j e^{i \theta_j})}{\partial \theta_j} d\theta_1 \ldots  d\theta_{j-1} d\theta_{j+1} \ldots d\theta_N.
\end{multline*}

Now note that the internal integral over $\theta_j$ is equal to $1$ if $0$ is inside the circle of center $B_j$ and radius $l_j$, in other words if $l_j>B_j$ and $0$ otherwise. So from this we deduce that

\begin{equation*}
\int_{\mathbb{T}^N} f(\theta) d\theta= \sum_{j=1}^{N} \omega_j \mathrm{mes}_{N-1} \left\{\theta: B(\theta_1,  \ldots, \theta_{j-1}, \theta_{j+1}, \ldots, \theta_N)<l_j \right\}.
\end{equation*}
\end{proof}
\section{Adapted proof: evaluation of the dipolar form.}\label{sec:3}
Let us consider a map $\arg: \tilde{\mathbb{C}} \rightarrow \mathbb{R}$  from the covering space of a punctured complex plane, $\tilde{\mathbb{C}}  \rightarrow\mathbb{C}^*$. This map gives an argument of a complex number different from $0$.  For any analytic curve $\gamma: \mathbb{R} \rightarrow \mathbb{C}$ on the plane the restriction of this \emph{argument map} on this curve $\gamma$ by $\arg_{\gamma} : \mathbb{R} \rightarrow \mathbb{R}$ gives a map that defines the argument $\arg \gamma (t)$ of the point on the curve. Each time we use this notation we suppose taking the continuous branch of the argument function (see the Remark \ref{remark:passage_through_zero} for the case when $\gamma$ passes through $0$).

For the case of Lagrange problem, we will be interested in taking as a curve $\gamma$ a trajectory $z(t)$ of the flow $\Psi \circ T^t$, as in \eqref{eq:z}. This trajectory can be seen as a map $z: \mathbb{R}_+ \rightarrow \mathbb{C}$. The map $\Psi: \mathbb{T}^N \rightarrow \mathbb{C}$ transports the singular $1$-form $d \arg z$ on the complex plane to a $1$-form on the torus that we will denote $\beta:=\Psi^* d\arg z$ and call the \emph{Lagrange form}.
This form $\beta$ is singular since $\Psi^{-1}(0) \neq \emptyset$. Indeed, for the case of three joints in the Lagrange problem, the set $\Psi^{-1}(0)$ corresponds to the set of $\bm{\theta}$ when the swiveling arm closes up into a triangle. In what follows, we will study regular and singular parts of Lagrange form $\beta$ and we will find a geometrical way to calculate its time average 
$\lim_{T \rightarrow \infty} \frac{1}{T} \int_0^T z^* \beta$. This time average can be seen as an average of the image of the form $\beta$ transported by the map $z$ but also it is exactly equal to the limit angular velocity $\omega$ we are interested in.

\subsection{Dipolar form and its properties.}\label{subsec:3.1}
In this Subsection we will first prove some statements about the integration of regular $1$-forms along the orbits of vector fields. Second, we will define a dipolar form $\beta_{\sing}$ on the torus - a specific singular form that will encode the singularities of the form $\beta$. We will see that the dipolar form contains all the important geometric information for the calculation of $\omega$. The idea is simple: the important changes of the argument occur only when the swiveling arm passes by zero. In other words, they occur when a trajectory of the vector field \eqref{eq:vfield} passes by the singularities of the dipolar form. 


\begin{lemma}\label{lemma:ergodic}

Consider a manifold $M$ with a measure $\mu$ on it and a uniquely ergodic flow $g^t: M \rightarrow M$ of a vector field $X$ on $M$, the measure $\mu$ being the only invariant measure. Then, the following assertions hold:
\begin{enumerate}
\item For any point $\bm{\theta} \in M$ and for any continuous function $f \in C^0 (M, \mathbb{R})$ there exists a limit of time averages $\lim_{T \rightarrow \infty} \frac{1}{T} \int_0^T f \circ g^t (\bm{\theta}) dt$ and this limit doesn't depend on the point $\bm{\theta} \in M$ and is equal to the space average $\int_M f d\mu$.
\item For $\bar{f}=f+X(h)$, where $h \in C^1(M, \mathbb{R})$ is any continuously differentiable function on $M$, the time average of $\bar{f}$ coincides with that of $f$.

\item For any closed $1$-form $\beta$ on $M$ define the function $f:=\beta(X)$. Then the space average $\int_{M} f d\mu$ depends only on the cohomology class of $\beta$.

\item Let  $M=\mathbb{T}^N$ and $X$ be given by \eqref{eq:vfield}. Then for any smooth $1$-form $\beta$ holds $\int_M \beta (X) = \left< [\beta] , [\omega_1, \ldots, \omega_N] \right>$. Here  $[\beta] \in H^1(\mathbb{T}^N, \mathbb{R})$ and $[\omega_1, \ldots, \omega_N] \in H_1(\mathbb{T}^N, \mathbb{R})$ denotes the sum of standard coordinate circles with coefficients $\omega_j \in \mathbb{R}$. We denote as $\left<\cdot, \cdot\right>$ the pairing between cohomology and homology. Note that $[\beta]$ has a representative $\beta_{\mathrm{reg}} \in [\beta]$ with constant coefficients $\beta_j \in \mathbb{R}$: $\beta_{\mathrm{reg}}=\sum_{j=1}^N \beta_j d \theta_j$ and $\int_M \beta(X)=\sum_{j=1}^N \beta_j \omega_j$.
\end{enumerate}
\end{lemma}

\begin{proof}
\begin{enumerate}
\item The existence of the limit and its independence from the initial point $\bm{\theta} \in M$ follows from Birkhoff's ergodic theorem. 

\item The difference between time averages of $f$ and $\bar{f}$ can be rewritten by Newton-Leibniz. Since $g$ is a bounded function, we obtain
\begin{equation}\label{eq:proof}
\lim_{T \rightarrow \infty} \frac{1}{T} \int_0^T X(h) \circ g^t (\bm{\theta}) dt=\lim_{T \rightarrow \infty}  \frac{h(g^T(\bm{\theta}))-h(\bm{\theta})}{T} =0.
\end{equation}
\item We have to prove that the space average $\int_M \beta(X) d \mu$ doesn't change if $\beta$ is replaced by $\bar{\beta}=\beta+dh$ where $h \in C^1(M, \mathbb{R})$. This can be deduced from $(2)$: indeed, the space average $\int_M \bar{\beta}(X) d \mu$ is equal to the corresponding time average (by ergodic theorem), and then one applies \eqref{eq:proof} to finish the argument.

\item The first statement is the application of $(3)$ to this particular case $M=\mathbb{T}^n$, $X=\sum_j \omega_j \frac{\partial}{\partial \theta_j}$. Each form $\beta \in H^1(\mathbb{T}^N, \mathbb{R})$ has a representative with constant coefficients since $H^1(\mathbb{T}^N, \mathbb{R}) \cong \mathbb{R}^N$. And hence $\int_{\mathbb{T}^N} \beta (X)$ for a smooth form $\beta$ is equal to the corresponding value for its cohomology representative with constant coefficients. $\square$

\end{enumerate}
\end{proof}

Now let is fix two distinct points $a,b \in \mathbb{C}$. Let us consider a following multifunction $f$ on the complex plane: $f(z)=\arg \frac{z-a}{z-b}$. This multifunction can not be defined on all of the plane in a continuous way although it is well defined outside a large enough ball $B(R)=\{ |x| \leq R\}$ containing $a$ and $b$, see Figure \ref{pic1}.

\begin{figure}
\begin{center}
\includegraphics*[scale=0.8]{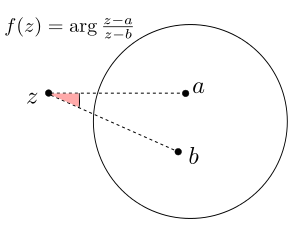}
\caption[]{For the multivalued function $f(z)=\arg \frac{z-a}{z-b}$ outside the big ball $B(R)$ containing points $a$ and $b$ one can define a continuous determination of $f$ as an angle between two rays connecting $z$ to $a$ and to $b$ correspondingly.}\label{pic1}
\smallskip
\end{center}
\end{figure}

Let us choose a function $\bar{f}: \mathbb{C} \rightarrow \mathbb{R}$ such that $\bar{f}=f$ in $\mathbb{C} \setminus B(R)$ and $\bar{f} \in C^{\infty}$. Then $h=f-\bar{f}$ is a multifunction such that $h=0$ in $\mathbb{C} \setminus B(R)$.

\begin{definition}[Dipolar form]
A dipolar form is a singular $1$-form $dh$ on the complex plane.
\end{definition}

\subsection{New proof of Theorem \ref{thm:euclidian_3}.}\label{subsec:3.2}
Let us consider the Lagrange form $\beta$ on the torus: our goal is to understand its time average along the orbits of a linear flow $T^t$ on the torus $\mathbb{T}^N$. What was said before in this Section, can be applied to any dimension but from now on we will study the particular case $N=3$. First of all, by Lemma \ref{lemma:rotating}, one can reduce dimension to $2$ and suppose that the system is governed by the field \eqref{eq:vfield} with $\omega_1=0$. 

From now on we will look at the map $\Psi$ as at the map from a $2$-torus to $\mathbb{C}$, and the Lagrange form $\beta$ will be considered as a form on a $2$-torus as well (we will speak about the \emph{reduced Lagrange form} in this case). This torus $\mathbb{T}^2$ is equipped with coordinates $(\theta_2, \theta_3)$ that correspond to the angles that the second and the third joint make with a horizontal direction. 

As we have already seen in the proof of Section \ref{sec:2} as well as in the Subsection \ref{subsec:3.1} of this Section, the important increments of the argument of $z(t)$ are those corresponding to the passages through zero. In other words, the singular set $\Psi^{-1}(0)$ is of importance in the Lagrange problem. In the case when $\omega_1=0$ the set $\Psi^{-1}(0)$ consists of two different points $A, B \in \mathbb{T}^2$ that correspond to the positions of the swiveling arm depicted on Figure \ref{pic2}. One can note that \begin{equation}\label{eq:ABcoordinates}
A=(-\pi+\alpha_3, \pi-\alpha_2),\; 
B=(\pi-\alpha_3, \pi+\alpha_2).
\end{equation}

\begin{figure}
\centering
\includegraphics*[scale=1]{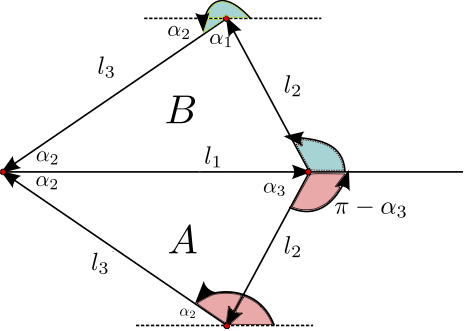}
\caption[]{Two positions of the swiveling arm of type $(l_1, l_2, l_3)$ corresponing to the situations when this swiveling arm forms a triangle. These two positions correspond to the points $A,B \in \mathbb{T}^2$ that have the following coordinates $(\theta_2, \theta_3) \in \mathbb{T}^2$: $A=(-\pi+\alpha_3, \pi-\alpha_2)$ and $B=(\pi-\alpha_3, \pi+\alpha_2)$.
These coordinates are the counter-clockwise oriented angles that the joints of the arm form with horizontal direction. They are explicitely marked on the picture.}\label{pic2}
\smallskip
\end{figure}

Now, the dipolar form that we defined on $\mathbb{C}$ in Subsection \ref{subsec:3.1} can be transported to a $1$-form on $\mathbb{T}^2$ in such a way that its singularities $a,b$ are transported to the points $A,B \in \mathbb{T}^2$. For this, we will choose a disk on the torus containing the points $A,B$ and transport the dipolar form on the plane to the form that we denote $\beta_{\mathrm{sing}}$. 

\begin{remark}
This dipolar form on the torus depends on the choice of the disk containing $A,B \in \mathbb{T}^2$. We will fix this choice as shown on the Figure \ref{pic:dipolar}.
\end{remark}

\begin{figure}
\begin{center}
\includegraphics*[scale=0.7]{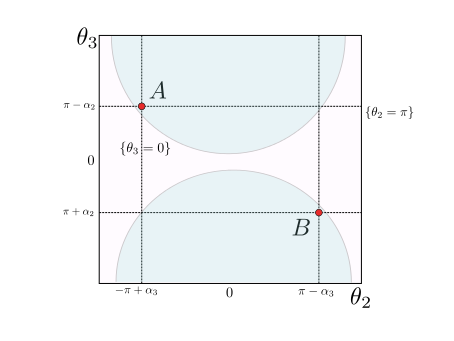}
\caption[]{The torus $\mathbb{T}^2$ of positions $(\theta_2, \theta_3)$ of a swiveling arm of type $(l_1, l_2, l_3)$ for the movement in the vector field \eqref{eq:vfield} with $\omega_1=0$. Points $A,B \in \mathbb{T}^2$ correspond to the positions when the arm forms a triangle. Here $\alpha_2, \alpha_3$ are the corresponding angles of this triangle. A choice of a disk containing the points $A,B$ fixes a dipolar form $\beta_{\sing}$ on $\mathbb{T}^2$ with two logarithmic singularities.}
\smallskip\label{pic:dipolar}
\end{center}
\end{figure}

Then we have a following

\begin{lemma}\label{lemma:decomposition}
Consider the dynamics of a swiveling arm of type $(l_1, l_2, l_3)$ with $l_j, j=1,2,3$ satisfying all three strict triangle inequalities, in a vector field \eqref{eq:vfield} with $\omega_1=0$.
Let $A,B \in \mathbb{T}^2$ be as in \eqref{eq:ABcoordinates} and let us fix a choice of a dipolar form $\beta_{\sing}$ (depending on a disc containing $A,B \in \mathbb{T}^2$) in $\mathbb{T}^2$ as defined above. Then there exists a unique form   
%
%
$\beta_{\mathrm{reg}} \in H^1(\mathbb{T}^2, \mathbb{R})$ with constant coefficients and a function $f \in C^{1}(\mathbb{T}^2)$ such that $\beta=\beta_{\mathrm{reg}}+\beta_{\mathrm{sing}}+df$. 
\end{lemma}

\begin{remark}
Different choice of a circle containing $A,B$ would provoke a different form $\beta_{\sing}$, and hence, different form $\beta_{\reg}$. 
\end{remark}

\begin{proof}
First, $\delta:=\beta-\beta_{\mathrm{sing}}$ is a smooth $1$-form on the torus. Indeed, when a point $\bm{\theta} \in \mathbb{T}^2$ makes a loop around the point $A$ (respectively, $B$) on the torus, the argument of the end of the swiveling arm grows (or, respectively, diminishes) by $2 \pi$ exactly as a value of the dipolar form. This means that the points $A, B \in \mathbb{T}^2$ can't be the singularities of $\delta$ nor can be any other point. This form $\delta$ has its representative $\beta_{\mathrm{reg}}$ in a family of forms with constant coefficients since $H^1(\mathbb{T}^N, \mathbb{R}) \cong \mathbb{R}^N$. Hence $\delta-\beta_{\mathrm{reg}}$ is a differential of a smooth function.
\end{proof}

Now we are ready to give a new proof of Theorem \ref{thm:euclidian_3}.
\begin{proof}
Suppose that $\omega_1=0$. Then the Lagrange problem is equivalent to the study of the time average of the reduced Lagrange form $\beta$ along the orbits of the reduced vector field $X_{\red}:=\omega_2 \frac{\partial}{\partial d \theta_2}+\omega_3 \frac{\partial}{\partial d \theta_3}$. This time average by Lemma \ref{lemma:decomposition} is a sum of time averages for $\beta_{\sing}, \beta_{\reg}$ and $df, f \in C^1(\mathbb{T}^2)$. For the last one, the part $(1)$ of Lemma \ref{lemma:ergodic} gives that the time average of $df$ along the flow is equal to the space average which is zero by Stokes Theorem since $\partial \mathbb{T}^2=0$.

\textbf{Step 1. Calculate the time average of the regular part.} 

Following the part $(4)$ of Lemma \ref{lemma:decomposition} we see that the time average of $\beta_{\reg}=\beta_2 d \theta_2 + \beta_3 d \theta_3, \beta_2, \beta_3 \in \mathbb{R}$ is its evaluation on the reduced vector field $X_{\red}$. As already noticed in Lemma \ref{lemma:decomposition}, $\beta_{\reg}$ depends on a choice of a topological disk containing points $A$ and $B$ or, equivalently, on the choice of the homotopy path $\gamma$ connecting $A$ and $B$. The disk was fixed once and for all once we defined $\beta_{\sing}$, see Figure \ref{pic3}. Let us choose the generators of cohomology $H^1(\mathbb{T}^2, \mathbb{R})$ in such a way that they do not intersect this disk.

We choose these paths as shown on Figure \ref{pic3}: one of them is horizontal and another one is vertical. 

Geometrically, $\beta_2$ corresponds to the increment of $\arg z(t)$ when $\theta_3=0$ and $\theta_2$ makes one turn. In this case, the argument doesn't change because of triangle inequality, $|l_2|<|l_1|+|l_3|$ and the turning second vector will never get around $0$ if the first and the third one are pointing in one direction, see Figure \ref{pic:period1}. Analogously, $\beta_3=1$ because the argument changes by $2 \pi$ when the third interval is making one turn and the second is fixed, pointing in the direction $\theta_2=\pi$. Hence the time average of the regular part of Lagrange form is equal to $\left<\beta_{\mathrm{reg}} ,[\omega_2, \omega_3]\right>=\omega_3$.

\begin{figure}
\begin{center}
\includegraphics*[scale=0.9]{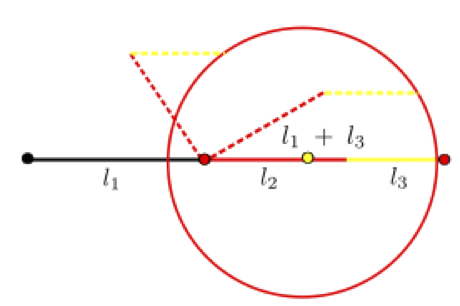}
\includegraphics*[scale=0.9]{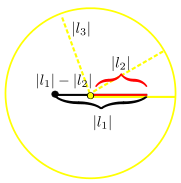}
\caption[]{For $\theta_1=0$ fixed, the movement on a torus $\mathbb{T}^2$ corresponding to the circle $\{\theta_3=0\}$ is described by the picture on the left. The second joint makes a circle movement: in this case the end of the system makes a circle movement as well, and this is a circle with the center $l_1+l_3$ and radius $l_2$. One can easily notice that this circle can't contain $0$ if the triangle inequality $l_2<l_1+l_3$ holds. The picture on the right describes the increment of the argument along the circle $\{\theta_2=\pi\}$: analogously, the end of the system moves along the circle with the center $l_1-l_2$ and radius $l_3$. In this case, on the contrary, this circle contains $0$. }\label{pic:period1}
\smallskip
\end{center}
\end{figure}


\begin{figure}
\begin{center}
\includegraphics*[scale=0.7]{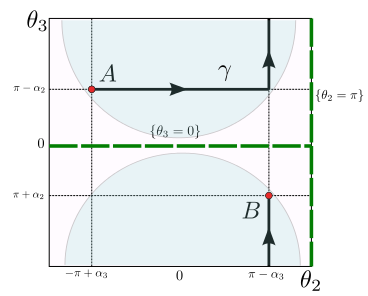}
\caption[]{One can choose a path $\gamma$ connecting the points $A$ and $B$ on the two-torus as shown on the picture. This path consists of one horizontal and one vertical part which correspond to the complete rotation of the second joint and then, to the complete rotation of the third joint to reach $B$ from  $A$. This path is contained in the disk that was chosen previously for the definition of the dipolar form $\beta_{\mathrm{sing}}$. Any path between $A$ and $B$ in this disk has the same homotopy type as $\gamma$. The flux of the vector field $X_{\red}$ through this path is equal to the evaluation of $\beta_{\sing}$ on $X$. The circles $\{\theta_3=0\}$ and $\{\theta_2=\pi\}$ are chosen as generators of $H^1(\mathbb{T}^2, \mathbb{R})$ that do not intersect $\gamma$ in order to define $\beta_{\reg}$ correctly.}\label{pic3}
\smallskip
\end{center}
\end{figure}

\bigskip

\textbf{Step 2. Calculate the time average of the dipolar part.}
Consider a path $\gamma$ connecting the points $A$ and $B$ that is chosen on the Figure \ref{pic3} and contained in the disk where the dipolar form is non-zero. Note that all the paths inside this disk joining $A$ and $B$ are homotopic (as paths with fixed extremities). The important observation is that the time average of the dipolar form is equal to the flux of the vector field $X$ through this path. The intuition behind this statement is that the argument of $\arg z(t)$ changes by $2 \pi$ (grows or diminishes in dependence of the direction) only if the trajectory $z(t)$ crosses the path between $A$ and $B$. A formal argument is the following.

Consider a rectangle which is obtained from $\gamma$ when pushing with $g^{\varepsilon}$, see Figure \ref{fig:pathdipolar}. The flux of the vector field is the area of this rectangle. We can apply the ergodic theorem to this rectangle (since its boundary has measure zero) to get that the flux of $X$ is equal to the time average of dipolar form almost everywhere. To get that the needed limit exists everywhere (and not almost everywhere), we use the fact that the linear flow on the torus is equicontinuous (and even more, it preserves distances). The points which are close to each other will meet the rectangle $R$ in close points (the exceptions exist but are very rare, see Figure \ref{fig:pathdipolar}).

\begin{figure}
\begin{center}
\includegraphics*[scale=0.7]{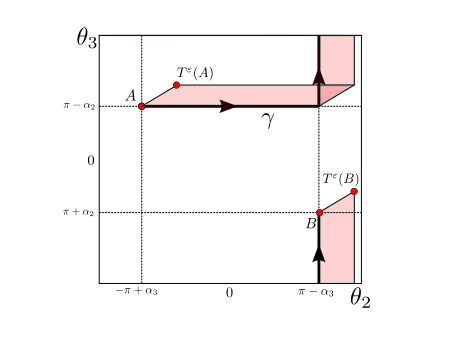}
\caption[]{The calculation of the flux of $X$ through $\gamma$ is equivalent to the calculation of the area of the rectangle $R$ defined as a set covered by the trajectories of the flow $g^t$. The points $A$ and $B$ are marked as well as their images by $g^{\varepsilon}$: $g^{\varepsilon}(A)$ and $g^{\varepsilon}(B)$. }\label{fig:pathdipolar}
\smallskip
\end{center}
\end{figure}

What is left is a calculation of the flux of the vector field $X=[\omega_2, \omega_3]$ through $\gamma$. On the first segment of the path when $\theta_3$ remains constant and equal to $\pi-\alpha_2$, the flux depends only on the vertical component of the field \eqref{eq:vfield}, $\omega_3$. The trajectories of $X$ are transverse to the path and intersect it from the left to the right, so the flux 
on this interval of the path is equal to $- \frac{2 \pi - 2 \alpha_3}{2 \pi } \omega_3$. Analogously, the flux through the vertical component of the path is equal to $\frac{2 \alpha_2}{2 \pi} \omega_2$.

We calculated the time average of the dipolar part. Let us note that the dipolar part and the regular part are intimately related.
An important remark about the calculation of the periods of a regular part of the form $\beta$ is the following. The numbers $\beta_2, \beta_3$ calculated above are the periods of the form $\beta_{\mathrm{reg}}$. To calculate them, we integrate this form on the paths in $\mathbb{T}^2$ which correspond to the first and second generator of cohomology $H^1(\mathbb{T}^2, \mathbb{R})$. What is important is that those paths are chosen in a way not to intersect the path $\gamma$ that is connecting the singularities. Only in this case the evaluation of a regular part will give us the correct quantity corresponding to the time average of the form $\beta-\beta_{\mathrm{sing}}$.
\bigskip

\textbf{Step 3. Sum them up.}
By adding up the evaluations of $\beta_{\reg}$ and $\beta_{\sing}$, we obtain: $\omega=\frac{\alpha_2}{\pi} \omega_2 + \frac{\alpha_3}{\pi} \omega_3$ in the case when $\omega_1=0$. By passing back to the system where $\omega_1 \neq 0$, see Lemma \ref{lemma:rotating}, we obtain the answer in the general case:
\begin{equation*} 
\omega=\omega_1+\frac{\alpha_2}{\pi} (\omega_2-\omega_1)+\frac{\alpha_3}{\pi} (\omega_3-\omega_1)=\sum_{j=1}^{3} \frac{\alpha_j}{\pi} \omega_j.
\end{equation*}
\end{proof}

\section{Non-zero curvature.}\label{sec:4}
Lagrange problem can be considered on any riemannian surface $S$ which is \emph{oriented} (in order to define the angular velocities and rotations) and \emph{complete} (in order to be able to connect the points on this surface by geodesic paths).

Indeed, let us fix some point $x_0 \in S$ and fix the lengths $l_j \in \mathbb{R}_+, \omega_j \in \mathbb{R}_+, j=1, \ldots, N$ and $\bm{\theta}^{(0)}=\left(\theta_1^{(0)}, \ldots, \theta_N^{(0)} \right) \in \mathbb{T}^N$. We will define the dynamics of a swiveling arm of type $(l_1, \ldots, l_N)$ based at $x_0$ under the flow of the vector field $\eqref{eq:vfield}$ given by $\omega_j$ with the initial condition defined by $\bm{\theta}^{(0)}$. Let us proceed as follows.
 
Choose an angle coordinate on the fiber of unitary tangent bundle $T^1_{x_0} \cong \mathbb{S}^1$. Consider a geodesic interval of length $l_1$ coming out from $x$ in the direction equal to  $\theta_1^{(0)}+ \omega_1 t$. Then in its endpoint $x_1$ the circle $T^1_{x_1} S$ has a privileged point (corresponding to the continuation of the movement along the geodesic). Then, one can define a geodesic interval of length $l_2$ emanating from $x_1 \in S$ in the direction equal to $\theta_2^{(0)}+  \omega_2 t$ counted from this privileged point and so on. The ending point $x_n$ of such a construction is called the end of the swiveling arm of type $(l_1, \ldots, l_N)$ on the riemannian surface $S$ at time $t$ under the flow of the vector field \eqref{eq:vfield}. This ending point defines a curve $z(t): \mathbb{R}_+ \rightarrow S$. See Figure \ref{pic:nonconstant}.

\begin{figure}
\begin{center}
\includegraphics*[scale=0.7]{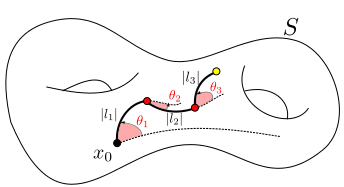}
\caption[]{A swiveling arm on the oriented compete surface $S$ of non-constant curvature.
}\label{pic:nonconstant}
\smallskip
\end{center}
\end{figure}

\begin{definition}
Suppose that there exists a complex chart on the surface $S$ such that the curve $\{z(t)| t \in \mathbb{R}\}$ is contained in a bounded ball $B(x_0,R)$. 
 \emph{The Lagrange problem on the oriented and complete surface} $S$ is a study of the existence of the limit \eqref{eq:angular} in this chart as well as its value as a function of $l_j \in \mathbb{C}, \omega_j \in \mathbb{R}$ and initial condition $\bm{\theta}^{(0)}$ .
\end{definition}

\begin{remark}
If the lengths $l_j, j=1, \ldots, N$ are small enough then an open chart (such that the corresponding complex structure is compatible with the metric, and hence the angles can be measured accordingly) in the definition of the Lagrange problem on $S$ exists.
\end{remark}

\subsection{Constant curvature Lagrange problem: redefining the angles.}\label{subsec:4.1}
Note that there is an important difference between the definition of the Lagrange problem on a general surface we have given above and the definition of Lagrange problem on the plane given in Subsection \ref{subsec:1.2}. Indeed, the plane has a specialty of having a globally defined horizontal direction and the angles $\theta_j$ for the Lagrange problem on the plane are measured with respect to this direction. Since on the general surface a choice of such a direction is impossible, the angle coordinates  $\theta_j$ of the swiveling arm are measured with respect to the positions of previous joints, see Figure \ref{pic:nonconstant}. For the euclidian plane these two sets of coordinates are related in an obvious way by a linear transformation.

\begin{proposition}\label{prop:new_angles}
Consider a swiveling arm on $\mathbb{R}^2$ with $N$ joints . Suppose that $(\theta^h_1, \ldots, \theta^h_N) \in \mathbb{T}^N$ are the angles that the joints make with the horizontal direction and  $(\theta_1, \ldots, \theta_N) \in \mathbb{T}^N$ are the angles that the joints make with the direction of the previous joint in the system. Then those two sets are related by a following linear relation:

\bigskip

\(
\begin{pmatrix}
\theta_1\\
\theta_2\\
\vdots \\
\vdots \\
\theta_N
\end{pmatrix}
= \begin{pmatrix} 
    1 & 0 & 0& \dots & 0 \\
    1 & -1 & 0& \dots & 0 \\
        \vdots & \vdots & \ddots & \ddots & \vdots \\
    0& 1&-1 & \ldots& 0  \\ 
    0& 0&  \ldots & 1 & -1   
    \end{pmatrix}
\begin{pmatrix}
\theta^h_1\\
\theta^h_2\\
\vdots \\
\vdots \\
\theta^h_N
\end{pmatrix}.
\)
\bigskip

Consequently, if one replaces the coordinates $\theta_j^h$ by the coordinates $\theta_j$, the resulting limit velocity in Theorem \ref{thm:euclidian_3} is equal to 
\begin{equation*}
\omega=\frac{\alpha_1}{\pi} \omega_1+\frac{\alpha_2}{\pi}(\omega_1+\omega_2)+\frac{\alpha_3}{\pi} (\omega_1+\omega_2+\omega_3)=
\omega_1+\omega_2 \frac{\alpha_2+\alpha_3}{\pi}+\omega_3\frac{\alpha_3}{\pi}.
\end{equation*}
\end{proposition}

\begin{proof}
Straightforward, see Figure \ref{pic:newangles}.
\end{proof}

\begin{figure}
\centering
\includegraphics*[scale=0.8]{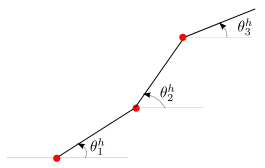}
\includegraphics*[scale=0.8]{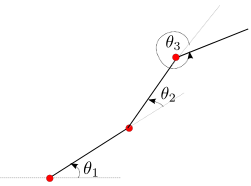}
\caption[]{Different ways to define the coordinates $\theta_j \in \mathbb{S}^1, j=1, \ldots, N$: on the left, with respect to the common horizontal direction, and on the right, with respect to the previous joint.}\label{pic:newangles}
\smallskip
\end{figure}

\begin{theorem}\label{thm:constant_curvature}
Consider a Lagrange problem on a constant curvature surface $S$ which is either the sphere $\mathbb{S}^2$ of radius $1$ or the hyperbolic plane $\mathbb{H}^2$ for $N=3$. For a swiveling arm with $N=3$ joints of type $(l_1, l_2, l_3)$ based at a point $x_0 \in S$ and the flow of vector field $\eqref{eq:vfield}$  suppose the following:
\begin{enumerate}
\item $l_j, j=1,2,3$ satisfy all three strict triangle inequalities, 
\item $\omega_1, \omega_2, \omega_3 \in \mathbb{R}_+$ are rationally independent.
\end{enumerate} 
In the case $S=\mathbb{S}^2$ suppose in addition that $\sum_{j=1}^3{l_j}<\pi $. 

Then, there exists a triangle $\Delta$ on $S$ with the lengths of sides equal to $l_j, j=1,2,3$. Denote its angles correspondingly $\alpha_j, j=1,2,3$ (its angles are uniquely defined by the lengths of its sides). Then, the asymptotic angular velocity $\omega$ exists and is equal to 
\begin{equation}\label{eq:thesame}
\omega=\omega_1+\frac{\pi-\alpha_1}{\pi}\omega_2 +\frac{\alpha_3}{\pi} \omega_3.
\end{equation}
\end{theorem}
\begin{remark}
The result of this Theorem can be rewritten in the terms of the area $A$ of the triangle $\Delta$. Indeed, the formula \eqref{eq:thesame} is equivalent to 
\begin{equation}\label{eq:thesame_area}
\omega=\omega_1+\frac{\alpha_2+\alpha_3 \pm A}{\pi}\omega_2 +\frac{\alpha_3}{\pi} \omega_3.
\end{equation}
for the hyperbolic ($+A$) and spherical ($-A$) cases. 

Note that the answer given by \eqref{eq:thesame} is a general answer for all constant curvature geometries, 
see
Proposition \ref{prop:new_angles} for the euclidian case. 
\end{remark}

\begin{proof}
We can suppose that $\omega_1=0$ since the argument of Lemma \ref{lemma:rotating} still works for spherical and hyperbolic geometry in which Lagrange problem has rotational symmetry. Consider a movement of the swiveling arm in the reduced vector field $\omega_2 \frac{\partial}{\partial \theta_2}+ \omega_3 \frac{\partial}{\partial \theta_3}$. Here $\theta_j$ are the new coordinates defined in the beginning of this Section corresponding to the angles between the direction of a joint number $j$ and the direction of a previous joint in a swiveling arm. 

Here we will simply repeat the proof of Theorem \ref{thm:euclidian_3} from Subsection \ref{subsec:3.2} modulo some minor changes. All the notions are defined analogously: Lagrangian $1$-form $\beta$, its regular and singular (dipolar) parts, $\beta_{\reg}$ and $\beta_{\sing}$. The only difference is that the coordinates $\theta_j, j=2,3$ on the torus $\mathbb{T}^2$ are not the same as before (see Figure \ref{pic:newangles}) so one has to recalculate the evaluations of $\beta_{\sing}$ and $\beta_{\reg}$ but the geometrical essence of the argument doesn't change. 

Note that if the lengths of the joints verify three strict triangle inequalities and if, in the case $S=\mathbb{S}^2$, the sum of the lengths is smaller than the distance between the north and south poles, there exists a triangle with the sides of lengths $l_j$, uniquely defined up to isometry. We denote $\alpha_j$ its angles. Then, the coordinates of singularities of $\beta$ change : we replace the Figure \ref{pic2} by the Figure \ref{pic:lobachevsky}. One can see that now the singularities have the following coordinates: $A(-\pi+\alpha_3, -\pi+\alpha_1)$ and $B(\pi-\alpha_3, \pi-\alpha_1)$. A path $\gamma$ from $A$ to $B$ is chosen in a way shown on the Figure \ref{pic:path_lobachevsky} (analogue of Figure \ref{pic3}).
%

Then, the evaluation of Lagrange $1$-form is a sum of the evaluations of singular and regular parts, the evaluation of a singular part will give
%
\begin{equation}\label{eq:evaluation_of_a_singular_part}
-\frac{2 \pi - 2 \alpha_3}{2 \pi} \omega_3 - \frac{2 \alpha_1}{2 \pi} \omega_2.
\end{equation}

The regular part with constant coefficients can be written as $\beta_{\mathrm{reg}}=\beta_2 d \tilde{\theta}_2 + \beta_3 d \tilde{\theta}_3$ and by calculating its periods, one obtains $\beta_2=\beta_3=1$.

By adding the evaluations of $\beta_{\reg}, \beta_{\sing}$ and $\omega_1$ (which signifies the returning back to the initial system where the first joint turns), one gets the final answer.
%
\begin{figure}
\centering
\includegraphics*[scale=0.7]{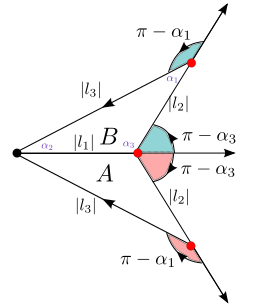}
\caption[]{
Two positions corresponding to the singular points $A,B$ of the dipolar form $\beta_{\mathrm{sing}}$ (and, accordingly, Lagrange form $\beta$) on $S=\mathbb{H}^2$ (or $S=\mathbb{S}^2$). These positions correspond to a swiveling arm that closes up into a triangle with the sides of lengths $l_j, j=1,2,3$ and the angles of values $\alpha_j, j=1,2,3$. This permits to calculate the coordinates of $A, B$ which are defined as angles between the present direction of the joint and the positive direction of the previous joint. We suppose that the coordinate is growing when the angle changes counterclockwise.} \label{pic:lobachevsky}
\smallskip
\end{figure}

\begin{figure}
\centering
\includegraphics*[scale=0.7]{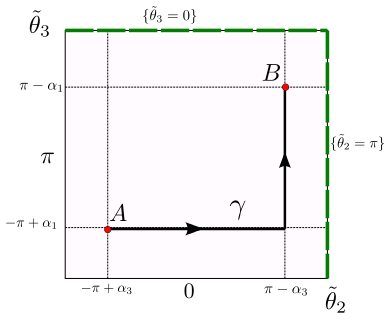}
\caption[]{
A path of integration $\gamma$ for a singular part $\beta_{\sing}$ of Lagrange form. The green paths $\left\{
\tilde{\theta}_3=0 \right\}$ and $\tilde{\theta}_2=\pi$ are useful for the calculation of the periods of the regular part $\beta_{\reg}$ of Lagrangian $1$-form.
} \label{pic:path_lobachevsky}
\smallskip
\end{figure}
\end{proof}

\subsection{Non-constant curvature: kite property.}\label{subsec:4.2}
In this Subsection we will solve the Lagrange problem  on a \emph{non-constant} curvature surface $S$
for a swiveling arm with $3$ joints based at some point $x_0 \in S$ .

The two main obstructions for the argument that we elaborated for the constant curvature case are the following:
\begin{enumerate}
\item The geometry on the arbitrary riemmanian surface $S$ is not isotropic : for a fixed base point $x_0$ the geometry in different directions in $T^1_{x_0} S$ varies. This means that it won't be possible to restrict ourselves to the case $\omega_1=0$ since the Lagrange problem doesn't have a rotational symmetry.
\item For three positive numbers $l_1, l_2, l_3$ that satisfy all three of the strict triangle inequalities there is no guaranty that the triangles with such lengths of sides are all isometric, and hence, have the same angles. And, moreover, if one fixes a position $I \subset S$, $x_0$ of a first joint on the surface $S$, one doesn't guaranty that there are only two positions of a swiveling arm that closes up in a triangle with one of the sides coinciding with $I$ as on Figures  \ref{pic2} and \ref{pic:lobachevsky}.

\end{enumerate}

We were able to overcome the first obstruction by considering the Lagrange form as a form on $\mathbb{T}^3$ and not on $\mathbb{T}^2$ as before. The second one is much trickier and we restrict ourselves to the case when it doesn't cause any problems: in the case when the lengths of the joints are small enough.

\begin{definition}[Kite property for the oriented and complete surface $S$.]
Fix a triple of three positive numbers $(l_1, l_2, l_3) \in \mathbb{R}_+^3$, verifying all of three strict triangle inequalities. Consider an orientable complete riemannian surface $S$ with a point $x_0 \in S$ on it. The surface $S$ verifies a \emph{kite property} in the point $x_0 \in S$ for the triple $(l_1, l_2, l_3)$ if for any direction $\varphi \in T_{x_0}^1 S$ there exist two triangles $\Delta^+$ and $\Delta^-$ on $S$ with the sides of lengths $l_1, l_2, l_3$ such that 
\begin{enumerate}
\item[•] $\Delta^+$ and $\Delta^-$ have a common vertex in $x_0$ 
\item[•] For both $\Delta^+$ and $\Delta^{-}$, the side of length $l_2$ doesn't contain $x_0$
\item[•] $\Delta^+ \cap \Delta^-$ is a segment on the surface of length $l_1$ coinciding with one of the sides of both triangles and this segment lies on a geodesic $\gamma$ going out from $x_0$ in the direction $\varphi$ : $\gamma(0)=0, \dot{\gamma}(0)=\varphi$.

\item[•] The couple of triangles $(\Delta^+, \Delta^-)$ is a unique couple with the properties listed above.
\end{enumerate}
\end{definition}

We fix the notations by saying that $\Delta^+$ ($\Delta^-$, correspondingly) is a triangle which is lying on the left (on the right) from the geodesic associated to $(x_0, \varphi) \in T^1 S$.

\begin{figure}
\centering
\includegraphics*[scale=0.7]{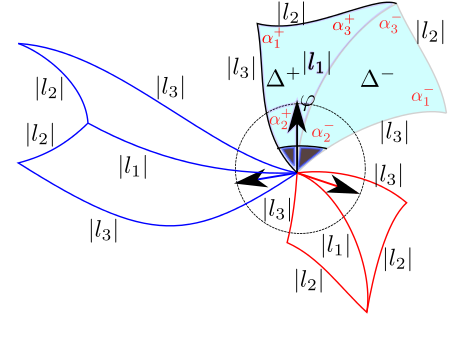}
\caption[]{Fix a direction $\varphi \in T^1_{x_0} S$ and consider a geodesic ray emanating from $x_0 \in S$ in this direction. By the $(l_1, l_2, l_3))$-kite property of the surface $S$ in the point $x_0$, one can exhibit two triangles $\Delta^+(\varphi)$ and $\Delta^-(\varphi)$ glued one to each other along the side of length $l_1$, forming a kite. These kites change their forms while $\varphi$ varies in $T^1_{x_0}$. For example, their angles depend on $\varphi$ as well. }\label{pic:1234567}
\smallskip
\end{figure}

\begin{remark}
For the proofs of Theorem \ref{thm:euclidian_3} and Theorem \ref{thm:constant_curvature} we use nothing more than a kite property for $S= \mathbb{R}^2, \mathbb{H}^2$ or $\mathbb{S}^2$. On a general surface $S$ there is no hope for the kite property to hold for any triple of lengths.
\end{remark}

\begin{proposition}\label{prop:kiteproperty}
Fix a complete oriented riemannian surface $S$. Then there exists some constant $C(S)>0$ such that $\forall x_0 \in S$ the kite property for swiveling arms based at $x_0$ holds for all triples $l=(l_1, l_2, l_3) \in \mathbb{R}_+^3$ such that their lengths are small enough, $|l|_{\infty}=\max_{j} l_j \leq C(S)$.
\end{proposition}
\begin{proof}
This follows from the convexity of small discs: there exists a uniform constant $C(S)>0$ such that all the disks of radii smaller than $C(S)$ are strictly convex, \cite{Sakai}.
Take a triple $(l_1, l_2, l_3)$ in such a way that $|l|_{\infty} \leq C$. Let us fix $\varphi \in T^1_{x_0} S$ and construct a unique geodesic $\gamma$ from the Definition of kite property: $\gamma(0)=x_0, \dot{\gamma}=\varphi$. Let $x_1:=\gamma(l_1)$. Consider two disks : $B(x_0, l_3)$ and $B(x_1, l_2)$. By convexity, they will intersect in exactly two points, see Figure \ref{pic:123123}.

\end{proof}

\begin{figure}
\centering
\includegraphics*[scale=0.7]{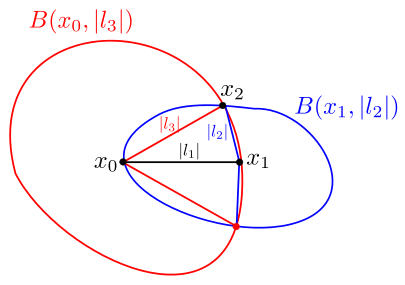}
\caption[]{There are two points in the intersection of two circles which are the boundaries of convex balls $B(x_0, l_3)$ and $B(x_1, l_2)$. These two points correspond to the positions of a swiveling arm that closes up into a triangle with the of lengths $l_1, l_2, l_3$.}\label{pic:123123}
\smallskip
\end{figure}


\begin{theorem}\label{thm:non-constant_curvature}
Consider the Lagrange problem on an arbitrary oriented and complete riemmanian surface $S$ for a swiveling arm with $N=3$ joints of type $(l_1, l_2, l_3)$ based a some point $x_0 \in S$ and the flow of vector field $X=\sum_{j=1}^3 \omega_j \frac{\partial}{\partial \theta_j}$, see $\eqref{eq:vfield}$. Suppose the following:
\begin{enumerate}
\item $l_j, j=1,2,3$ satisfy all three strict triangle inequalities 
\item $|l|_{\infty}:=\max_j (l_j) \leq C(S, x_0)$ where $C$ is a constant from Proposition \ref{prop:kiteproperty}
\item $\omega_1, \omega_2, \omega_3 \in \mathbb{R}_+$ are rationally independent.
\end{enumerate} 

Then, for any $\varphi \in T_{x_0}^1 S$ there exist the triangles $\Delta^+(\varphi)$ and $\Delta^-(\varphi)$ with the properties described above in the definition of kite property. 
Denote the angles of these triangles correspondingly $\alpha^{\pm}_1(\varphi), \alpha^{\pm}_2(\varphi), \alpha^{\pm}_3(\varphi)$.

Then, the asymptotic angular velocity $\omega$ exists and is equal to 
\begin{equation*}
\omega=
\omega_1+\omega_2  \frac{\pi - \bar{\alpha}_1}{\pi} + \omega_3 \frac{\bar{\alpha}_3}{\pi},
\end{equation*}
\begin{equation*}
\textit{where}  \;\; \bar{\alpha}_j=\frac{\bar{\alpha}_j^{+}+\bar{\alpha}_j^{-}}{2} \; \; \textit{and} \; 
\bar{\alpha}_j^{\pm}=\frac{1}{2 \pi} \int_{T_{x_0}^1 S} \alpha_j^{\pm} (\varphi) \; d \varphi , j=1,2,3.
\end{equation*}
Here $\bar{\alpha}_j^{\pm}$ are the average values of the absolute values of the angles in triangles with the sides of lengths $l_j$ in the kite property, see Pic. \ref{pic:1234567} with respect to the direction of the first interval. Here the parameter  $\varphi$ comes from the definition of a kite property.
\end{theorem}

\begin{proof}
The idea is to adjust the proofs from Subsections \ref{subsec:3.2} and \ref{subsec:4.1} that dealt with constant curvature to a non-constant curvature case. We will still consider the Lagrange $1$-form $\beta$ and its regular and singular parts $\beta_{\reg}$ and $\beta_{\sing}$ but they are now all $1$-forms on $\mathbb{T}^3$ and not $\mathbb{T}^2$.
The singular set $\mathcal{S}$ of $\beta_{\sing}$ (and, respectfully, $\beta$) has changed: for each plane $\Pi_{\varphi}:=\{\theta_1=\varphi \in T_{x_0} S\}$ the intersection of $\mathcal{S}$ with this plane consists of two points that correspond to the positions of a swiveling arm closing up into a triangle:
\begin{equation*}
\mathcal{S} \cap \Pi_{\varphi}=\left\{A(\varphi), B(\varphi)\right\}.
\end{equation*}
These points exist since the kite property holds, see Proposition \ref{prop:kiteproperty}, and their coordinates can be represented as
\begin{align*}
A(\varphi)=\left(-\pi+\alpha_3^{-}(\varphi), -\pi+\alpha_1^{-} (\varphi) \right),\\
B(\varphi)=\left(
\pi-\alpha_3^{+}(\varphi), \pi-\alpha_1^{+}(\varphi)
\right),
\end{align*}
by the same argument as in the proof of Theorem \ref{thm:constant_curvature}. Hence, for smal $|l|_{\infty}$ the singular set $\mathcal{S}$ is a union of two circles. The asymptotic velocity $\omega$ is given by the evaluation $\beta[X]$ which is the sum of two numbers: the evaluation of the regular part and that of the singular parts. The first one is a space integral and the second one is a flux through a surface with a boundary $\mathcal{S}$, i.e. a cylinder, see Figure \ref{fig:cylinder}. We will represent it as a union of paths with fixed $\theta_1$.

\begin{figure}
\centering
\includegraphics*[scale=0.7]{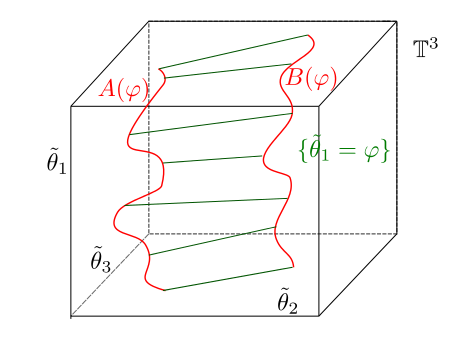}
\caption[]{This figure represents a fondamental domain of the torus $\mathbb{T}^3$ . The evaluation of $\beta_{\sing}$ is equal to the flux of $X$ through the surface of the cylinder foliated by intervals $\theta_1 = \mathrm{const}$.}\label{fig:cylinder}
\smallskip
\end{figure}

Fix $\theta_1=\varphi$. Since $\{A(\varphi) | \varphi \in T_{x_0}^1S\}$ and $\{B(\varphi) | \varphi \in T_{x_0}^1 S \}$ are two closed circles, the $\tilde{\theta}_1$-component of the vector field $X$ won't give any contribution to the evaluation of a singular part. Then, the flux is calculated exactly as in \ref{eq:evaluation_of_a_singular_part}. Taking into account that $\alpha^+(\varphi) \neq \alpha^-{\varphi}$ we have that
%
the evaluation of $\beta_{\sing}$ on the vector field $X$ for $\theta_1=\varphi$ gives
\begin{equation*}
-\frac{2 \pi - \alpha_3^+(\varphi)-\alpha_3^-(\varphi)}{2 \pi} \omega_3 - \frac{ \alpha_1^+(\varphi)+\alpha_1^-(\varphi)}{2 \pi} \omega_2.
\end{equation*}

By integration over $T_{x_0}^1 S$, we obtain

\begin{equation*}
\left<\beta_{\mathrm{sing}}, [X] \right>=
-\frac{\pi-\bar{\alpha}_3}{\pi}\omega_3-\frac{\bar{\alpha}_1}{\pi}\omega_2.
\end{equation*}

The evaluation of the regular part $\beta_{\reg}=\beta_1 d \tilde{\theta}_1 + \beta_2 d \tilde{\theta}_2 + \beta_3 d \tilde{\theta}_3, \beta_1, \beta_2, \beta_3 \in \mathbb{R}$ is given by its periods $\beta_1, \beta_2, \beta_3$ that we can calculate by integrating $\beta_{\reg}$ on three circles: correspondingly, $\{(\varphi, \pi, 0) | \varphi \in \mathbb{S}^1\}, 
\{(0, \varphi, 0) | \varphi \in \mathbb{S}^1\}$ and $\{(0, \pi, \varphi) | \varphi \in \mathbb{S}^1\}$.

Each one of these circles is disjoint from the cylinder of singularities, Moreover, there is a torus containing this cylinder disjoint from these three circles. This is clear for the two last paths since $\theta_1=\mathrm{const}$ and this follows from the $2$-dimensional pictures drawn before, see for example Figure \ref{pic:path_lobachevsky}. The first circle neither doesn't intersect the cylinder since this corresponds to a degenerate position that is never approached by continuous curves $\{A(\varphi)\}$ and $\{B(\varphi)\}, \varphi \in T_{x_0}^1 S$. One can easily see that in all of three cases, $\beta_j=1, j=1,2,3$ and hence $\left<\beta_{\mathrm{reg}},[X]\right>= \omega_1+\omega_2+\omega_3$. By summing up two contributions we get the final answer. 

\end{proof}

\end{document}